\documentclass[a4paper,11pt,reqno,english]{smfart}

\usepackage{enumerate}
\usepackage{amssymb,amsmath,latexsym,amsthm}
\usepackage[T1]{fontenc}
\usepackage[a4paper, tmargin=1.5in, bmargin=1.5in, lmargin=1in,rmargin=1in]{geometry}
\usepackage{url}
\usepackage[frenchb, main=english]{babel}
\usepackage[utf8]{inputenc}
\usepackage{mathrsfs}
\usepackage{esint}
\usepackage{xcolor}
\usepackage{bigints}
\usepackage{comment}
\usepackage[Kac]{dynkin-diagrams}
\definecolor{violet}{rgb}{0.0,0.2,0.7}
\definecolor{rouge2}{rgb}{0.8,0.0,0.2}
\usepackage{empheq}
\usepackage{tikz-cd}
\usetikzlibrary{matrix,arrows,decorations.pathmorphing}
\usepackage{hyperref}
\usepackage{mathpazo} 
\hypersetup{
	bookmarks=true,         
	unicode=false,          
	pdftoolbar=true,        
	pdfmenubar=true,        
	pdffitwindow=false,     
	pdfstartview={FitH},    
	pdftitle={},    
	pdfauthor={},     
	colorlinks=true,       
	linkcolor=rouge2,          
	citecolor=violet,        
	filecolor=black,      
	urlcolor=cyan}           
\usepackage{enumitem}
\usepackage{appendix}

\usepackage{mathtools}
\usepackage{dsfont}

\theoremstyle{plain}
    \newtheorem{thm}{Theorem}[section]
	
	\newtheorem{lem}[thm]{Lemma}
	\newtheorem{prop}[thm]{Proposition}
	
	\newtheorem{cor}[thm]{Corollary}
	
\theoremstyle{plain}
	\newtheorem{bigthm}{Theorem}
	\newtheorem{bigprop}[bigthm]{Proposition}
	\newtheorem{bigcor}[bigthm]{Corollary}
	
    \newtheorem*{bigrmk*}{Remark}

\theoremstyle{definition}
	\newtheorem{defn}[thm]{Definition}
	\newtheorem{eg}[thm]{Example}

	\newtheorem*{claim*}{Claim}
        
	\newtheorem*{ack*}{Acknowledgements}
\theoremstyle{remark}
	\newtheorem{rmk}[thm]{Remark}
	\newtheorem*{rmk*}{Remark}

        \newtheorem{ques}[thm]{Question}
	\newtheorem*{ques*}{Question}
	\newtheorem*{ans*}{Answer}

\numberwithin{equation}{section}
\renewcommand{\labelenumi}{(\roman{enumi})}
\newlist{steps}{enumerate}{1}
\setlist[steps, 1]{label = Step \arabic*:}

\mathcode`l="8000
\begingroup
\makeatletter
\lccode`\~=`\l
\DeclareMathSymbol{\lsb@l}{\mathalpha}{letters}{`l}
\lowercase{\gdef~{\ifnum\the\mathgroup=\m@ne \ell \else \lsb@l \fi}}%
\endgroup

\DeclareFontFamily{U}{MnSymbolC}{}
\DeclareSymbolFont{MnSyC}{U}{MnSymbolC}{m}{n}
\DeclareFontShape{U}{MnSymbolC}{m}{n}{
	<-6>  MnSymbolC5
	<6-7>  MnSymbolC6
	<7-8>  MnSymbolC7
	<8-9>  MnSymbolC8
	<9-10> MnSymbolC9
	<10-12> MnSymbolC10
	<12->   MnSymbolC12}{}
\DeclareMathSymbol{\intprod}{\mathbin}{MnSyC}{'270}

\DeclareMathOperator{\supp}{supp}

\DeclareMathOperator{\length}{length}

\DeclareMathOperator{\mult}{mult}

\DeclareMathOperator{\ord}{ord}

\def\1{\mathds{1}}
\def\e{\mathrm{e}}

\def\RA{\mathrm{A}}
\def\RD{\mathrm{D}}
\def\RE{\mathrm{E}}

\def\L{\mathbf{L}}

\newcommand{\ii}{\mathrm{i}}

\newcommand{\loc}{\mathrm{loc}}
\newcommand{\fa}{\mathfrak{a}}
\newcommand{\fm}{\mathfrak{m}}

\newcommand{\wX}{{\widetilde{X}}}

\newcommand\vep{\varepsilon}
\newcommand\vph{\varphi}

\newcommand\om{\omega}
\newcommand\ta{\theta}
\newcommand\gm{\gamma}

\newcommand\RGL{\mathrm{GL}}
\newcommand\RSL{\mathrm{SL}}

\newcommand\BN{\mathbb{N}}
\newcommand\BZ{\mathbb{Z}}

\newcommand\BR{\mathbb{R}}
\newcommand\BC{\mathbb{C}}
\newcommand\BB{\mathbb{B}}
\newcommand\BS{\mathbb{S}}
\newcommand\BD{\mathbb{D}}
\newcommand\BP{\mathbb{P}}

\newcommand\CC{\mathcal{C}}

\newcommand\CI{\mathcal{I}}

\newcommand\CO{\mathcal{O}}

\newcommand\Fm{\mathfrak{m}}

\newcommand\lt{\left}
\newcommand\rt{\right}

\newcommand\pl{\partial}
\newcommand\db{\bar{\partial}}

\newcommand\ddb{\partial \bar{\partial}}

\newcommand\dd{\mathrm{d}}
\newcommand\dc{\mathrm{d}^{\mathrm{c}}}
\newcommand\ddc{\mathrm{d}\mathrm{d}^{\mathrm{c}}}

\newcommand\w{\wedge}

\newcommand\reg{\mathrm{reg}}
\newcommand\sing{\mathrm{sing}}

\newcommand\set[2]{\left\{ {#1} \; \middle\vert \; {#2} \right\}}

\newcommand{\RN}[1]{\textup{\uppercase\expandafter{\romannumeral#1}}}

\makeatletter
\newsavebox{\@brx}
\newcommand{\llangle}[1][]{\savebox{\@brx}{\(\m@th{#1\langle}\)}%
  \mathopen{\copy\@brx\kern-0.5\wd\@brx\usebox{\@brx}}}
\newcommand{\rrangle}[1][]{\savebox{\@brx}{\(\m@th{#1\rangle}\)}%
  \mathclose{\copy\@brx\kern-0.5\wd\@brx\usebox{\@brx}}}
\makeatother

\title{Demailly--Lelong numbers on complex spaces} 

\author{Chung-Ming Pan}
\address[Chung-Ming Pan]{Simons Laufer Mathematical Sciences Institute; 17 Gauss Way, Berkeley, CA 94720-5070, USA \qquad\qquad\qquad\qquad\qquad\qquad\qquad\qquad\qquad\qquad}
\email{\href{mailto:bandan770@gmail.com}{bandan770@gmail.com} \qquad\qquad\qquad\qquad\qquad\qquad\qquad\qquad\qquad\qquad\qquad\qquad\qquad\qquad}
\urladdr{\href{https://chungmingpan.github.io/}{https://chungmingpan.github.io/}}

\date{\today}

\begin{document} 

\maketitle

\begin{abstract}
We prove a conjecture proposed by Berman--Boucksom--Eyssidieux--Guedj--Zeriahi, affirming that the Demailly--Lelong number can be determined through a combination of intersection numbers given by the divisorial part of the potential and the SNC divisors over a log resolution of the maximal ideal of a given point. 
Moreover, this result establishes a pointwise comparison of two different notions of Lelong numbers of plurisubharmonic functions defined on singular complex spaces. 
We also provide an estimate for quotient singularities and sharp estimates for two-dimensional ADE singularities.
\end{abstract}

\section*{Introduction}
The Lelong number, introduced by Pierre Lelong \cite{Lelong_1957, Lelong_1968}, is a fundamental invariant in complex analysis and geometry (see \cite{Siu_1974, Demailly_1993, Demailly_agbook, GZbook} and the references therein). 
In complex geometry, singularities are ubiquitous across various domains within complex geometry, notably in Minimal Model Programs and compactifying moduli spaces. 
An important avenue of research involves exploring how fundamental objects and properties can be extended to singular complex spaces. 
In this note, our focus centers on the comparison of different notions of Lelong numbers for plurisubharmonic functions defined on singular complex spaces. 

\smallskip
In the complex Euclidean space $\BC^N$, a plurisubharmonic function is an upper semi-continuous function which is subharmonic along every complex line. 
On singular complex spaces, a plurisubharmonic function is a function that extends to a plurisubharmonic function defined near any local embedding into $\BC^N$.
For an $n$-dimensional 
locally irreducible reduced 
complex analytic space $X$, let $x \in X$ be a fixed point, and consider local generators $(f_i)_{i \in I}$ of the maximal ideal $\mathfrak{m}_{X,x}$ of $\CO_{X,x}$.
Set $\psi_x := \sqrt{\sum_{i \in I} |f_i|^2}$.
Let $\vph$ be a plurisubharmonic function defined near $x$.
From \cite[p. 45]{Demailly_1985}, the \emph{slope} of $\vph$ at $x$ is defined by
\[
    s(\vph,x) := \sup \set{\gm \geq 0}{\vph \leq \gm \log \psi_x + O(1)}.
\]
In \cite[D\'ef.~3]{Demailly_1982}, Demailly introduced another way of measuring the singularity of $\vph$ at $x$ by considering
\[
    \nu(\vph,x) := \lim_{r \to 0} {}^\downarrow \int_{\{\psi_x < r\}} (\ddc \vph) \w \lt(\ddc \log \psi_x \rt)^{n-1}.
\] 
We call it the \emph{Demailly--Lelong number} of $\vph$ at $x$. 
Notably, both of the slope and the Demailly--Lelong number are independent of the choice of $(f_i)_i$.

\smallskip
When $x$ is a smooth point in $X$, those two quantities are equal (see e.g. \cite[Thm.~2.32]{GZbook}). 
However, it is no longer the case when $(X,x)$ is singular. 
In such cases, one has the following inequality (cf. \cite[Rmk.~A.5]{BBEGZ_2019}) 
\begin{equation}\label{eq:slope<DL}
    \mult(X,x) \cdot s(\vph,x) \leq \nu(\vph,x),
\end{equation}
where $\mult(X,x) = \max\{k \in \BN \,\,|\,\, \mathscr{I}_{X,x} \subset \Fm_{\BC^N,x}^k\}$ is the multiplicity of $X$ at $x$, and $\mathscr{I}_X$ is the ideal sheaf of $X$ under a local embedding near $x$ into $\BC^N$.
The inequality \eqref{eq:slope<DL} is strict in general (cf. Proposition~\ref{bigprop:prop_C}).

\smallskip
Let $\pi': X' \to X$ be a normalization and let $\pi'': (\wX, \CO_\wX(-E)) \to (X', (\pi')^{-1}\Fm_{X,x})$\footnote{In the sequel, for any morphism $f: Y \to X$ and any ideal sheaf $\fa \subset \CO_X$, we shall consistently use the abuse of notation $f^{-1} \fa$ to represent the ideal sheaf $f^{-1} \fa \cdot \CO_Y$ of $\CO_Y$.} be a log-resolution of the ideal sheaf $(\pi')^{-1}\Fm_{X,x}$; namely, $(\pi'')^{-1} \lt((\pi')^{-1} \Fm_{X,x}\rt) = \CO_\wX(-E)$. 
Set $\pi := \pi' \circ \pi''$. 
A conjecture stated in {\it loc. cit.} proposes that the Demailly--Lelong number can be read as an intersection number $(D \cdot (-E)^{n-1})$, where $D$ is the divisorial part of $\ddc \vph \circ \pi$ over $\pi^{-1}(x)$ under (partial) Siu's decomposition. 
This would ensure the existence of a constant $C_x \geq 1$, independent of $\vph$, such that the following inequality holds:
\begin{equation}\label{eq:conj_comparison_Lelong_num}
    \nu(\vph,x) \leq C_x \cdot \mult(X,x) \cdot s(\vph,x).
\end{equation}

In this note, we confirm the aforementioned conjecture of Berman--Boucksom--Eyssidieux--Guedj--Zeriahi in \cite[Rmk.~A5]{BBEGZ_2019}. 
Precisely, we get
\begin{bigthm}\label{bigthm:intersection_and_comparison}
Let $(X,x)$ be a germ of $n$-dimensional 
locally irreducible reduced 
complex analytic space.
As above, take $\pi: (\wX, \CO_{\wX}(-E)) \to (X, \Fm_{X,x})$ a log-resolution of $\Fm_{X,x}$ where $E$ is effective and so that $\pi^{-1}\fm_{X,x} = \CO_\wX(-E)$\footnote{Note that $E$ may not represent the entire exceptional divisor of $\pi$.}.
Let $(E_i)_{i \in I}$ be the irreducible components of $E$. 
Then given any germ of plurisubharmonic functions $\vph: (X,x) \to \BR \cup \{-\infty\}$, the Demailly--Lelong number can be expressed as 
\[
    \nu(\vph,x) = (D \cdot (-E)^{n-1})
\]
where $D = \sum_i a_i E_i$ so that $\ddc \vph \circ \pi = \sum_{i \in I} a_i [E_i] + R$ with $\nu(R,E_i) = 0$ for all $i \in I$.
In particular, there is a constant $C_x \geq 1$ such that
\[
    \nu(\vph, x) \leq C_x \cdot \mult(X,x) \cdot s(\vph,x).
\]
for all germs of plurisubharmonic functions $\vph: (X,x) \to \BR \cup \{-\infty\}$.
\end{bigthm}

Combining \eqref{eq:slope<DL} and Theorem~\ref{bigthm:intersection_and_comparison}, if one of these Lelong numbers is zero, then so is another one:
\begin{bigcor}
Let $X$ be an $n$-dimensional 
locally irreducible reduced 
complex analytic space and let $x \in X$. 
Then given any germ of plurisubharmonic function $\vph: (X,x) \to \BR \cup \{-\infty\}$, we have
\[
    s(\vph,x) = 0 \iff \nu(\vph,x) = 0.
\]
\end{bigcor}

It is natural to wonder whether a uniform version of the comparison still holds; namely, a comparison independent of both the base point $x$ and the plurisubharmonic function $\vph$ (see Question~\ref{ques:general_comparison}).

\smallskip
The two-dimensional ADE singularities can be considered as the simplest singularities in complex and algebraic geometry.
They arise as an orbifold fixed point locally modeled by $\BC^2/ G$, where $G$ is a finite subgroup of $\RSL(2,\BC)$.
In order to further understand the behavior of $C_x$, we compute sharp estimates for two-dimensional ADE singularities: 
\begin{eqnarray*}
    \RA_k: & \quad (x^2 + y^2 + z^{k+1} = 0) \subset \BC^3 & \quad (k \in \BN_{\geq 1}),\\
    \RD_k: & \quad (x^2 + y^2 z + z^{k-1} = 0) \subset \BC^3 & \quad (k \in \BN_{\geq 4}),\\
    \RE_6: & \quad (x^2 + y^3 + z^4 = 0) \subset \BC^3, & \\
    \RE_7: & \quad (x^2 + y^3 + y z^3= 0) \subset \BC^3, & \\
    \RE_8: & \quad (x^2 + y^3 + z^5 = 0) \subset \BC^3. &
\end{eqnarray*}

\begin{bigprop}\label{bigprop:prop_C}
Let $(X,x)$ be a two-dimensional ADE-singularity. 
Then the following hold:
\begin{enumerate}
    \item if $(X,x) \simeq (\RA_k, 0)$, then $\nu(\vph,x) \leq \frac{k+1}{2} \cdot \mult(X,x) \cdot s(\vph,x)$ and this estimate is sharp;
    \item if $(X,x) \simeq (\RD_k, 0)$, $(\RE_6, 0)$, $(\RE_7, 0)$ or $(\RE_8, 0)$, then $\nu(\vph,x) = \mult(X,x) \cdot s(\vph,x)$.
\end{enumerate}
\end{bigprop}

Given that ADE singularities are a subset of quotient singularities, it is legitimate to extend our investigation to compare the slope and the Demailly--Lelong number on more general quotient singularities.
Precisely, we derive the following result: 

\begin{bigprop}\label{bigprop:quotient}
Suppose that $(X,x)$ is isomorphic to a quotient singularity $(\BC^n/G, \pi(0))$, where $G$ is a finite subgroup of $\RGL(n,\BC)$ and $\pi: \BC^n \to \BC^n/G$ is the quotient map.
Then for any germ for plurisubharmonic function $\vph: (X,x) \to \BR \cup \{-\infty\}$, 
\[
    \nu(\vph,x) \leq |G|^{n-1} \cdot s(\vph,x)
\]
where $|G|$ is the order of $G$.
\end{bigprop}

\subsection*{Organization of the article}
\begin{itemize}
    \item In Section~\ref{sec:slope_vs_DL}, we review the basic concepts of plurisubharmonic functions on singular complex spaces. 
    Then we give a proof of Theorem~\ref{bigthm:intersection_and_comparison}.
    
    \item Section~\ref{sec:ADE} is devoted to investigating the comparison on explicit examples of two-dimensional ADE singularities. We obtain the estimates in Proposition~\ref{bigprop:prop_C}.
    
    \item In Section~\ref{sec:rmk}, we first study the comparison on quotient singularities (Proposition~\ref{bigprop:quotient}), and we validate the sharpness of the estimates provided in Proposition~\ref{bigprop:prop_C} for $\RA_k$-singularities through a computation on the quotient chart. 
    We then consider a vertex of a cone over a smooth hypersurface. 
    At the end, we address some questions on the uniform comparison and maximum ratio between the Demailly--Lelong number and the slope and a remark in connecting to algebraic quantities.  
\end{itemize}

\begin{ack*}
The author is grateful to V.~Guedj and H.~Guenancia for their constant support and suggestions and to S.~Boucksom, M.~Jonsson, and M.~P\u{a}un for their interest in this work and helpful feedback. 
The author is indebted to A.~Patel, A.~Trusiani, and D.-V.~Vu for discussions and remarks, and to Q.-T.~Dang for carefully reading a first draft. 
The author would also like to thank the anonymous referee for useful comments and for pointing out the original reference on the extension theorem for plurisubharmonic functions on complex spaces.
This work partially benefited from research projects Paraplui ANR-20-CE40-0019 and Karmapolis ANR-21-CE40-0010.
\end{ack*}

\section{Slope versus Demailly--Lelong number}\label{sec:slope_vs_DL}

\subsection{Preliminaries}
In this section, we recall a few notions of pluripotential theory on singular spaces.
We set the twisted exterior derivative $\dc = \frac{\ii}{2\pi}(\db -\pl)$ so that $\ddc = \frac{\ii}{\pi} \ddb$.

\smallskip
In the sequel, we shall always suppose that $X$ is an $n$-dimensional reduced complex analytic space.
We denote by $X^\reg$ the complex manifold of regular points of $X$. 
The set of singular points
\[
    X^\sing := X \setminus X ^\reg
\] 
is a complex analytic subset of $X$ of complex codimension greater than or equal to $1$.  
By definition, for each point $x \in X$, there exist an open neighborhood $U$ of $x$ and a local embedding $j: U \hookrightarrow \BC^N$ for some $N \in \BN^\ast$.
The notion of smooth functions/forms on $X$ can be defined by setting as functions/forms on $X^\reg$ which extend smoothly under any local embedding. 
The operators $\pl, \db, \dd, \dc,$ and $\ddc$ are well-defined by duality. 
We refer to \cite{Demailly_1985} for a detailed presentation of these concepts. 

\smallskip
Similarly, one also has analytic notions of plurisubharmonic (psh) functions:

\begin{defn}
Let $u: X \to \BR \cup \{-\infty\}$ be a given function that is not identically $-\infty$ on any open subset of $X$. 
The function $u$ is psh on $X$ if it is locally the restriction of a psh function on a local embedding $X \underset{\loc.}{\hookrightarrow} \BC^N$.
\end{defn}

Forn{\ae}ss--Narasimhan \cite[Thm.~5.3.1]{FN_1980} proved that being psh on $X$ is equivalent to the following: for any analytic disc $h: \BD \to X$, $u \circ h$ is either subharmonic on $\BD$ or identically $-\infty$.

\smallskip
If $u$ is a psh function on $X^\reg$ and locally bounded from above on $X$, one can extend $u$ to $X$ as follows:
\begin{equation}\label{eq:usc_psh}
    u^\ast(x) = \limsup_{X^\reg \ni y \to x} u(y).
\end{equation}
Grauert--Remmert's extension theorem \cite[Satz~3]{Grauert_Remmert_1956} shows that the above extension is also a psh function on $X$ (see also \cite[Thm.~1.7]{Demailly_1985}): 

\begin{thm}
Suppose that $X$ is locally irreducible. 
If $u$ is psh on $X^\reg$ and locally bounded from above on $X$, then the function $u^\ast$ defined by \eqref{eq:usc_psh} is psh on $X$.
\end{thm}

\begin{rmk}\label{rmk:embedding_ref_potential}
Fix a point $x \in X$ and a local embedding $X \xhookrightarrow{\loc.} \BC^N$ near $x$ which sends $x$ to the origin $0 \in \BC^N$. 
Note that $\mathfrak{m}_{X,x} = \mathfrak{m}_{\BC^N,0} / \mathscr{I}_{X,0}$ where $\mathscr{I}_X$ is the ideal sheaf of $X \subset \BC^N$.
By Demailly's comparison theorem \cite[Thm.~4]{Demailly_1982}, one can check that the Demailly--Lelong number can also be expressed as 
\[
    \nu(\vph,x) = \lim_{r \to 0} {}^\downarrow \int_{\BB_r(0)} \ddc \vph \w (\ddc \log |z|)^{n-1} \w [X], 
\]
where $\BB_r(0)$ is a ball in $\BC^N$ with radius $r$ centered at $0$.
Similarly, one can also obtain
\[
    s(\vph,x) = \sup \set{\gm \geq 0}{\vph \leq \gm \lt(\log |z|\rt)_{|X} + O(1)}.
\] 
\end{rmk}

\subsection{Proof of Theorem~\ref{bigthm:intersection_and_comparison}}
Recall that $\pi': X' \to X$ is a normalization and $\pi'': \wX \to X'$ is a resolution of $X'$ such that $\pi^{-1}\Fm_{X,x} = \CO_{\wX}(-E)$ where $\pi = \pi' \circ \pi''$, $E = \sum_{i=1}^N m_i E_i$ is effective, $\supp(E)$ is an SNC divisor, and $E_i$'s are irreducible components of the exceptional locus.
We have $\ddc \log \psi_x = \sum_i m_i [E_i] + \ta$ where $\ta$ is a smooth semi-positive form.
By Siu's decomposition theorem, $\ddc \pi^\ast \vph = \sum_i a_i [E_i] + R$. 
Take $D := \sum_i a_i E_i$ and $\om$ a K\"ahler form on $\wX$. 

\smallskip
Before starting the proof of Theorem~\ref{bigthm:intersection_and_comparison}, we recall some useful lemmas and their proofs from \cite{BBEGZ_2019}: 

\begin{lem}[{\cite[p.~73]{BBEGZ_2019}}]
The slope can be expressed as $s(\vph,x) = \min_i (a_i/m_i)$.
\end{lem}
\begin{proof}
For each $i$, $\pi^\ast \vph$ has generic Lelong number $a_i$ along $E_i$; thus, one has
\[
    a_i = \sup\set{\gm \geq 0}{\pi^\ast\vph \leq \gm \log |s_i|_{h_i} + O(1)} 
\]
where $s_i$ is a section cutting out $E_i$ and $h_i$ is a hermitian metric on $\CO(E_i)$.
Note that
\[
    s(\vph,x) = \sup\set{\gm \geq 0}{\pi^\ast\vph \leq \sum_i \gm m_i \log |s_i|_{h_i} + O(1)}.
\]
This implies that $a_i \geq s(\vph,x) m_i$ for each $i$ and it finishes the proof.
\end{proof}

\begin{lem}[{\cite[p.~73]{BBEGZ_2019}}]\label{lem:psef}
For each $i$, $-D_{|E_i}$ is psef. 
\end{lem}

\begin{proof}
Let $T = \ddc \pi^\ast\vph$. 
We have $\{T\}_{BC|E_i} = 0$ where $\{\bullet\}_{BC}$ denotes the $(1,1)$ Bott--Chern class. 
Hence, $-D_{|E_i}$ is psef if and only if $\{R\}_{BC|E_i}$ is psef. 
By Demailly's regularization theorem, after slightly shrinking $\wX$, one can find a sequence of $(1,1)$-currents with analytic singularities $(R_k)_k$ converges weakly towards $R$ with the following properties:
\begin{itemize}
    \item $\{R_k\}_{BC} = \{R\}_{BC}$;
    \item $R_k \geq -\vep_k \om$ with $\vep_k > 0$ and $\vep_k \to 0$ as $k \to +\infty$;
    \item $R_k$ is less singular than $R$.
\end{itemize}
We have $\nu(R_k, E_i) = 0$ for all $i$.
Therefore, ${R_k}_{|E_i}$ is a well-defined closed $(1,1)$-current.
We get $(\{R\}_{BC} + \vep_k \{\om\}_{BC})_{|E_i}$ is psef for all $k$ and thus $\{R\}_{BC|E_i}$ is psef for any $i$.
\end{proof}

The strategy of the following lemma originally came from an Izumi-type estimate in \cite[Sec.~6.1]{Boucksom_Favre_Jonsson_2014}:
\begin{lem}[{\cite[Lem.~A.4]{BBEGZ_2019}}]\label{lem:BBEGZ_A4}
There exists a constant $C > 0$ such that 
\[
    \max_i (a_i/m_i) \leq C \min_i (a_i/m_i).
\]
\end{lem}
\begin{proof}
Since $X$ is locally irreducible, the normalization $\pi': X' \to X$ is a homeomorphism (cf. \cite[Ch. II, Cor.~7.13]{Demailly_agbook}). 
Set $y = (\pi')^{-1}(x)$.
By Zariski's main theorem, $(\pi'')^{-1}(y) = \pi^{-1}(x)$ is connected.  
Reordering $E_i$, we may assume that $a_1/m_1 = \min_i (a_i/m_i)_i$, $a_r/m_r = \max_i (a_i/m_i)_i$, and $E_i \cap E_{i+1} \neq \emptyset$ for all $i \in \{1,\cdots, r-1\}$.
From Lemma~\ref{lem:psef}, for all $i$, 
\[
    (-D_{|E_i}) \cdot (\om_{|E_i})^{n-2} 
    = - \sum_j a_j c_{i,j} \geq 0
\] 
where $c_{i,j} := (E_i \cdot E_j \cdot \om^{n-2})$.
Then we have $\sum_{i \neq j} a_j c_{i,j} \leq a_i |c_{i,i}|$.
Note that $c_{i,j} \geq 0$ if $j \neq i$ and $c_{i,i+1} > 0$ for all $i = 1, \cdots r-1$.
Hence, $a_{i+1} c_{i,i+1} \leq a_i |c_{i,i}|$ for all $i \in \{1, \cdots, r-1\}$.
All in all, one can deduce 
\[
    \max_i (a_i/m_i) 
    = \frac{a_r}{m_r} 
    \leq \lt(\prod_{i=1}^{r-1} \frac{m_i |c_{i,i}|}{m_{i+1} c_{i,i+1}}\rt) \frac{a_1}{m_1} = C \min_i (a_i/m_i)
\]
where $C := \prod_{i=1}^{r-1} \frac{m_i |c_{i,i}|}{m_{i+1} c_{i,i+1}}$. 
\end{proof}

We are now ready to prove Theorem~\ref{bigthm:intersection_and_comparison}:

\begin{proof}[Proof of Theorem~\ref{bigthm:intersection_and_comparison}]
The main strategy follows from a combination of the approximation approach as in \cite[Prop~4.2]{Pan_Trusiani_2023} and Lemma~\ref{lem:BBEGZ_A4}.
By \cite[Thm.~5.5]{FN_1980}, there is a sequence of smooth strictly psh functions $(\vph_j)_j$ decreasing towards $\vph$. 
Set 
\[
    L_{r, j,\vep_1,\cdots, \vep_{n-1}} = \int_{\{\psi_x < r\}} \ddc \vph_j \w \ddc \log (\psi_x + \vep_1) \w \cdots \w \ddc \log (\psi_x + \vep_{n-1}).
\]
Then $\nu(\vph, x) = \lim_{r \to 0} \lim_{j \to +\infty} \lim_{\vep_1 \to 0} \cdots \lim_{\vep_{n-1} \to 0} L_{r,j,\vep_1,\cdots,\vep_{n-1}}$. 
Recall that 
\[
    \ddc \pi^\ast \vph = \sum_i a_i [E_i] + R
    \quad\text{and}\quad
    \ddc \pi^\ast \log \psi_x = \sum_i m_i [E_i] + \theta 
\]
where $R$ is a current whose generic Lelong number along each $E_i$ is zero, and $\ta$ is a smooth semi-positive $(1,1)$-form. 
Pulling back the integration to $\wX$, we obtain
\begin{align*}
    \lim_{\vep_{n-1} \to 0} L_{r,j,\vep_1,\cdots,\vep_{n-1}}
    &=\lim_{\vep_{n-1} \to 0} 
    \int_{\pi^{-1}(\{\psi_x < r\})} \ddc \pi^\ast\vph_j \w \bigwedge_{k=1}^{n-1} \ddc \log (\pi^\ast \psi_x + \vep_k)\\
    &= \int_{\pi^{-1}(\{\psi_x < r\})} \ddc \pi^\ast\vph_j \w \bigwedge_{k=1}^{n-2} \ddc \log (\pi^\ast \psi_x + \vep_k) \w \lt(\sum_{i=1}^N m_i [E_i] + \ta\rt).
\end{align*}
Since $\pi^\ast \vph_j$ is constant along each $E_i$, $\int_{\pi^{-1}(\{\psi_x < r\})} \ddc \pi^\ast \vph \w \bigwedge_{k=1}^{n-2} \ddc \log (\pi^\ast \psi_x + \vep_k) \w [E_i] = 0$. 
Hence,
\[
    \lim_{\vep_{n-1} \to 0} L_{r,j,\vep_1,\cdots,\vep_{n-1}} 
    = \int_{\pi^{-1}(\{\psi_x < r\})} \ddc \pi^\ast\vph_j \w \bigwedge_{k=1}^{n-2} \ddc \log (\pi^\ast \psi_x + \vep_k) \w \ta
\]
and then inductively, we obtain 
\[
    \nu(\vph,x) = \lim_{r \to 0} \int_{\pi^{-1}(\{\psi_x < r\})} \lt(\sum_i a_i [E_i] + R\rt) \w \ta^{n-1} 
    = \sum_i a_i \int_{E_i} \ta^{n-1} + \lim_{r \to 0} \int_{\pi^{-1}(\{\psi_x < r\})} R \w \ta^{n-1}.
\]
Note that $R \w \ta^{n-1}$ puts no mass along $E_i$ for all $i$; indeed, $\1_{E_i} R = \nu(R,E_i) [E_i] = 0$. 
Hence, $\nu(\vph,x) = \sum_i a_i \int_{E_i} \ta_{|E_i}^{n-1}$ and this shows the intersection expression of the Demailly--Lelong number.

\smallskip
On the other hand, by Thie's theorem \cite{Thie_1967}, we have $\mult(X,x) = \nu(\log \psi_x, x)$ and thus, 
\[
    \mult(X,x) = \sum_i m_i \int_{E_i} \ta_{|E_i}^{n-1} 
    =\sum_i m_i e_i
\]
where $e_i := \int_{E_i} \ta^{n-1} = E_i \cdot \ta^{n-1} \geq 0$.
Then we have 
\begin{equation}\label{eq:Lelong_expression}
    \nu(\vph,x) = \sum_i \frac{a_i}{m_i} m_i e_i \leq \mult(X,x) \max_i (a_i/m_i).
\end{equation}
By Lemma~\ref{lem:BBEGZ_A4}, $\max_i (a_i/m_i) \leq C \min_i (a_i/m_i)$ and this completes the proof.
\end{proof}

\begin{rmk}
The equality $\nu(\vph, x) = \sum_i \frac{a_i}{m_i} m_i e_i$ in \eqref{eq:Lelong_expression} also implies 
\[
    \mult(X,x) \cdot s(\vph,x) = \lt(\sum_i m_i e_i\rt) \cdot \min(a_i/m_i) \leq \nu(\vph,x).  
\]
This shows an alternative proof of \eqref{eq:slope<DL}.
\end{rmk}

\section{Sharp estimates for two-dimensional ADE singularities}\label{sec:ADE}
We recall the formulas for two dimensional ADE singularities: 
\begin{eqnarray*}
    \RA_k: & \quad (x^2 + y^2 + z^{k+1} = 0) \subset \BC^3 & \quad (k \in \BN_{\geq 1}),\\
    \RD_k: & \quad (x^2 + y^2 z + z^{k-1} = 0) \subset \BC^3 & \quad (k \in \BN_{\geq 4}),\\
    \RE_6: & \quad (x^2 + y^3 + z^4 = 0) \subset \BC^3, & \\
    \RE_7: & \quad (x^2 + y^3 + y z^3= 0) \subset \BC^3, & \\
    \RE_8: & \quad (x^2 + y^3 + z^5 = 0) \subset \BC^3. &
\end{eqnarray*}
Note that the multiplicity at the origin is $2$ in each case. 
With very explicit resolution graphs for two-dimensional ADEs, we obtain the following result:

\begin{prop}[Proposition~\ref{bigprop:prop_C}]
Let $(X,x)$ be a two-dimensional ADE-singularity. 
Then the following hold:
\begin{enumerate}
    \item if $(X,x) \simeq (\RA_k, 0)$, then $\nu(\vph,x) \leq (k+1) s(\vph,x)$ and this estimate is sharp;
    \item if $(X,x) \simeq (\RD_k, 0)$, $(\RE_6, 0)$, $(\RE_7, 0)$ or $(\RE_8, 0)$, then $\nu(\vph,x) = 2 s(\vph,x)$.
\end{enumerate}
\end{prop}

\begin{rmk}
The resolution graph of two-dimensional ADE-singularities that we will use below can be found in some standard references of singularity theory and two-dimensional complex geometry (e.g. \cite[Sec.~4.2]{Kollar_Mori_1998}, \cite[Ch.~3, Sec.~3]{BHPV_book}). 
In each resolution graph below, we always have 
\[
    E_i \cdot E_j =
    \begin{cases}
        -2 & \text{if}\,\,\, i = j,\\
        1 & \text{if}\,\,\, i \neq j \,\,\,\text{and}\,\,\, E_i \cap E_j \neq \emptyset,\\
        0 & \text{if}\,\,\, E_i \cap E_j = \emptyset.
    \end{cases}
\]
In the graphs below, the numbers next to each $E_i$ denote the multiplicities (i.e. $m_i$) such that $\pi^{-1} \fm_{X,x} = \CO(-\sum_i m_i E_i)$. 
\end{rmk}

\subsection{$\RA_k$-singularities} 
For $k \geq 1$, an $\RA_k$-singularity is locally isomorphic to $(x^2 + y^2 + z^{k+1} = 0) \subset \BC^3$ near $0$.
The dual graph of the minimal resolution $\pi: (\wX, \CO(-E)) \to (\RA_k,\fm_0)$ is the following
\begin{center}
\begin{dynkinDiagram}[
         scale=1.8] A{oo.oo}    
\node[below] at (root 1) {1};
\node[above] at (root 1) {$E_1$};
\node[below] at (root 2) {1};
\node[above] at (root 2) {$E_2$};
\node[below] at (root 3) {1};
\node[above] at (root 3) {$E_{k-1}$};
\node[below] at (root 4) {1};
\node[above] at (root 4) {$E_{k}$};
\end{dynkinDiagram}
\end{center}
and 
$\pi^{-1} \Fm_{0} = \CO(-(E_1 + E_2 + \cdots + E_{k-1} + E_k))$.
If an effective divisor $D = \sum_{i=1}^k a_i E_i$ satisfies that $-D_{|E_i}$ is psef (and thus nef) for each $i$, then we have the following inequalities for the coefficients $(a_i)_i$:
\[
    \begin{cases}
        2a_1 - a_2 \geq 0,\\
        2a_i - a_{i-1} - a_{i+1} \geq 0 \qquad (2\leq i \leq k-1),\\ 
        2a_k - a_{k-1} \geq 0.\\
    \end{cases}
\]
By the above inequalities, one can deduce
\[
    \begin{cases}
        a_1 \geq a_2 - a_1 \geq a_3 - a_2 \geq \cdots \geq a_i - a_{i-1} \geq \cdots \geq a_k - a_{k-1},\\
        a_k \geq a_{k-1} - a_k \geq a_{k-2} - a_{k-1} \geq \cdots \geq a_{i-1} - a_i \geq \cdots \geq a_1 - a_2.
    \end{cases}
\]
Hence, $a_i \leq i a_1$ and $a_i \leq (k-i+1) a_k$ for all $1 \leq i \leq k$.
Let $\vph$ be a psh function defined near $(\RA_k,0)$. 
Suppose that $\ddc \pi^\ast \vph = [D] + R$ with $\nu(R,E_i) = 0 $ for all $i$. 
Then we obtain
\[
    s(\vph,0) = \min\{a_1, a_2, \cdots, a_k\}
    = \min\{a_1, a_k\},
\]
and 
\begin{equation}\label{eq:A_k_est}
    \nu(\vph,0) = a_1 + a_k \leq (k+1) s(\vph,0).
\end{equation}
To check the estimate \eqref{eq:A_k_est} is sharp, we refer to examples in Example~\ref{eg:A_k_sharp}. 

\subsection{$\RD_k$-singularities}
For a $k \geq 4$, a $\RD_k$-singularity is locally isomorphic to $(x^2 + y^2 z + z^{k-1} = 0) \subset \BC^3$ near $0$.
The dual graph of the minimal resolution $\pi: (\wX, \CO(-E)) \to (\RD_k,\fm_0)$ is 
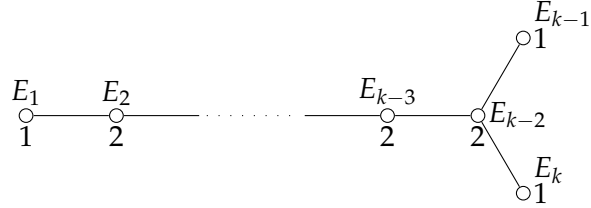
\begin{center}
\begin{dynkinDiagram}[
         scale=1.8] D{oo.oooo}    
\node[below] at (root 1) {1};
\node[above] at (root 1) {$E_1$};
\node[below] at (root 2) {2};
\node[above] at (root 2) {$E_2$};
\node[below] at (root 3) {2};
\node[above] at (root 3) {$E_{k-3}$};
\node[below] at (root 4) {2};
\node[right] at (root 4) {$E_{k-2}$};
\node[right] at (root 5) {1};
\node[above right] at (root 5) {$E_{k}$};
\node[right] at (root 6) {1};
\node[above right] at (root 6) {$E_{k-1}$};
\end{dynkinDiagram}
\end{center}
and 
$\pi^{-1} \Fm_0 = \CO(-(E_1 + 2E_2 + \cdots + 2E_{k-3} + 2E_{k-2} + E_{k-1} + E_k))$.
If $D = \sum_{i=1}^{k} a_i E_i$ is effective and $-D_{|E_i}$ is psef for each $i$, then the coefficients satisfying 
\[
\begin{cases}
    2a_1 - a_2 \geq 0,\\
    2a_i - a_{i-1} - a_{i_1} \geq 0 \qquad (2\leq i \leq k-3),\\
    2a_{k-2} - a_{k-3} - a_{k-1} - a_{k} \geq 0, \\
    2a_{k-1} - a_{k-2} \geq 0, \\
    2a_k - a_{k-2} \geq 0. 
\end{cases}
\]
Let $\vph$ be a psh function defined near $(\RD_k, 0)$.
Suppose that $D$ is the divisorial part of $\ddc \pi^\ast \vph$ over the origin. 
By direct computations, one has
\[
    \mult(\RD_k,0) \cdot s(\vph,0) 
    = 2 \min\lt\{a_1, \frac{a_2}{2}, \cdots, \frac{a_{k-2}}{2}, a_{k-1}, a_k\rt\}
    = a_2
\]
and 
\begin{align*}
    \nu(\vph,0) = -D \cdot E 
    &= 2(a_1 + 2a_2 + \cdots + 2a_{k-3} + 2a_{k-2} + a_{k-1} + a_k) \\
    &\qquad - (2a_1 + 3a_2 + 4a_3 + \cdots + 4a_{k-2} + 2a_{k-1} + 2a_k)\\
    &= a_2 = \mult(\RD_k,0) \cdot s(\vph,0).
\end{align*}

\subsection{$\RE_6$, $\RE_7$, $\RE_8$-singularities}
We now discuss the $\RE_6$, $\RE_7$, $\RE_8$-singularities case by case. 

\subsubsection{$\RE_6$-singularity}
An $\RE_6$ singularity is locally isomorphic to $(x^2 + y^3 + z^4 = 0) \subset \BC^3$ and it has the following dual graph for its minimal resolution $\pi: (\wX, \CO(-E)) \to (\RE_6, \fm_0)$:
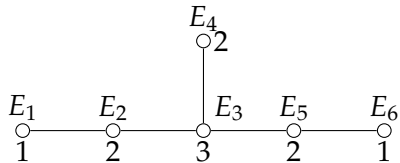
\begin{center}
\begin{dynkinDiagram}[
         scale=1.8] E{6}    
\node[below] at (root 1) {1};
\node[above] at (root 1) {$E_1$};
\node[below] at (root 2) {2};
\node[above] at (root 2) {$E_2$};
\node[below] at (root 3) {3};
\node[above right] at (root 3) {$E_3$};
\node[below] at (root 4) {2};
\node[above] at (root 4) {$E_5$};
\node[below] at (root 5) {1};
\node[above] at (root 5) {$E_6$};
\node[right] at (root 6) {2};
\node[above] at (root 6) {$E_4$};
\end{dynkinDiagram}
\end{center}
and 
$\pi^{-1} \Fm_0 = \CO(-(E_1 + 2E_2 + 3E_3 + 2E_4 + 2E_5 + E_6))$.
If $D = \sum_{i=1}^6 a_i E_i$ is effective and $-D_{|E_i}$ is psef for each $i$, then 
\[
\begin{cases}
    2a_1 - a_2 \geq 0,\\
    2a_2 - a_1 - a_3 \geq 0,\\
    2a_3 - a_2 - a_4 - a_5 \geq 0,\\
    2a_4 - a_3 \geq 0,\\
    2a_5 - a_3 - a_6 \geq 0,\\
    2a_6 - a_5 \geq 0. 
\end{cases}
\]
Let $\vph$ be a psh function defined near $(\RE_6, 0)$ such that $D$ is the divisorial part of $\ddc \pi^\ast \vph$ over the origin. 
One can infer
\[
    \mult(\RE_6,0) \cdot s(\vph,0) 
    = 2 \min\lt\{a_1, \frac{a_2}{2}, \frac{a_3}{3}, \frac{a_4}{2}, \frac{a_5}{2}, a_6\rt\}
    = \min\lt\{2a_1, a_2, \frac{2a_3}{3}, a_4, a_5, 2a_6\rt\}
    = a_4
\]
and 
\begin{align*}
    \nu(\vph,0) = -D \cdot E 
    &= 2(a_1 + 2a_2 + 3a_3 + 2a_4 + 2a_5 + a_6) \\
    &\qquad - (2a_1 + 4a_2 + 6a_3 + 3a_4 + 4a_5 + 2a_6)\\
    &= a_4 = \mult(\RE_6,0) \cdot s(\vph,0).
\end{align*}

\subsubsection{$\RE_7$-singularity}
An $\RE_7$ singularity is locally isomorphic to $(x^2 + y^3 + yz^3 = 0) \subset \BC^3$ and it has the following dual graph for its minimal resolution $\pi: (\wX, \CO(-E)) \to (\RE_7, \fm_0)$:
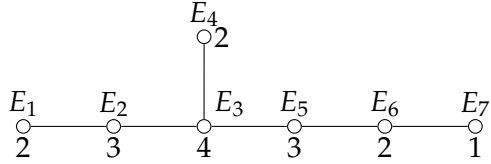
\begin{center}
\begin{dynkinDiagram}[
         scale=1.8] E{7}    
\node[below] at (root 1) {2};
\node[above] at (root 1) {$E_1$};
\node[below] at (root 2) {3};
\node[above] at (root 2) {$E_2$};
\node[below] at (root 3) {4};
\node[above right] at (root 3) {$E_3$};
\node[below] at (root 4) {3};
\node[above] at (root 4) {$E_5$};
\node[below] at (root 5) {2};
\node[above] at (root 5) {$E_6$};
\node[below] at (root 6) {1};
\node[above] at (root 6) {$E_7$};
\node[right] at (root 7) {2};
\node[above] at (root 7) {$E_4$};
\end{dynkinDiagram}
\end{center}
and 
$\pi^{-1} \Fm_0 = \CO(-(2E_1 + 3E_2 + 4E_3 + 2E_4 + 3E_5 + 2E_6 + E_7))$.
Again, if $D = \sum_{i=1}^{7} a_i E_i$ is effective and $-D_{|E_i}$ is psef for each $i$, then we have the following inequalities for the coefficients:
\[
\begin{cases}
    2a_1 - a_2 \geq 0,\\
    2a_2 - a_1 - a_3 \geq 0,\\
    2a_3 - a_2 - a_4 - a_5 \geq 0,\\
    2a_4 - a_3 \geq 0,\\
    2a_5 - a_3 - a_6 \geq 0,\\
    2a_6 - a_5 - a_7 \geq 0,\\
    2a_7 - a_6 \geq 0.
\end{cases}
\]
Let $\vph$ be a psh function defined near $(\RE_7,0)$ and $D$ is the divisorial component of $\ddc \pi^\ast \vph$ along the exceptional divisor. 
One can compute
\[
    \mult(\RE_7,0) \cdot s(\vph,0) 
    = 2 \min\lt\{\frac{a_1}{2}, \frac{a_2}{3}, \frac{a_3}{4}, \frac{a_4}{2}, \frac{a_5}{3}, \frac{a_6}{2}, a_7\rt\}
    = a_1
\]
and 
\begin{align*}
    \nu(\vph,0) = -D \cdot E
    &= 2(2a_1 + 3a_2 + 4a_3 + 2a_4 + 3a_5 + 2a_6 + a_7) \\
    &\qquad - (3a_1 + 6a_2 + 8a_3 + 4a_4 + 6a_5 + 4a_6 + 2a_7)\\
    &= a_1 = \mult(\RE_7,0) \cdot s(\vph,0).
\end{align*}

\subsubsection{$\RE_8$-singularity}
An $\RE_8$ singularity is locally isomorphic to $(x^2 + y^3 + z^5 = 0) \subset \BC^3$, and it has the following dual graph for its minimal resolution $\pi: (\wX, \CO(-E)) \to (\RE_8, \fm_0)$:
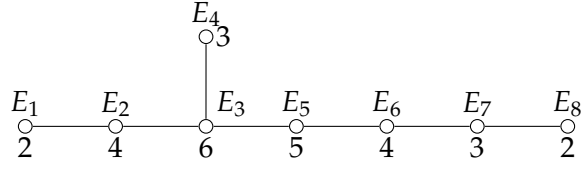
\begin{center}
\begin{dynkinDiagram}[
         scale=1.8] E{8}    
\node[below] at (root 1) {2};
\node[above] at (root 1) {$E_8$};
\node[below] at (root 2) {3};
\node[above] at (root 2) {$E_7$};
\node[below] at (root 3) {4};
\node[above] at (root 3) {$E_6$};
\node[below] at (root 4) {5};
\node[above] at (root 4) {$E_5$};
\node[below] at (root 5) {6};
\node[above right] at (root 5) {$E_3$};
\node[below] at (root 6) {4};
\node[above] at (root 6) {$E_2$};
\node[below] at (root 7) {2};
\node[above] at (root 7) {$E_1$};
\node[right] at (root 8) {3};
\node[above] at (root 8) {$E_4$};
\end{dynkinDiagram}
\end{center}
and $\pi^{-1} \Fm_0 
= \CO(-(2E_1 + 4E_2 + 6E_3 + 3E_4 + 5E_5 + 4E_6 + 3E_7 + 2E_8))$.
Let $D$ be an effective divisor with $-D_{|E_i}$ being psef for all $i$. 
Then 
\[
\begin{cases}
    2a_1 - a_2 \geq 0,\\
    2a_2 - a_1 - a_3 \geq 0,\\
    2a_3 - a_2 - a_4 - a_5 \geq 0,\\
    2a_4 - a_3 \geq 0,\\
    2a_5 - a_3 - a_6 \geq 0,\\
    2a_6 - a_5 - a_7 \geq 0,\\
    2a_7 - a_6 - a_8 \geq 0,\\
    2a_8 - a_7 \geq 0.
\end{cases}
\]
Take $\vph$ a psh function defined near $(\RE_8, 0)$ and assume that $D$ is the divisorial part of $\ddc \pi^\ast \vph$ over $0$.
This yields
\[
    \mult(\RE_8,0) \cdot s(\vph,0) 
    = 2 \min\lt\{\frac{a_1}{2}, \frac{a_2}{4}, \frac{a_3}{6}, \frac{a_4}{3}, \frac{a_5}{5}, \frac{a_6}{4}, \frac{a_7}{3}, \frac{a_8}{2}\rt\}
    = a_8
\]
and 
\begin{align*}
    \nu(\vph,0) = -D \cdot E 
    &= 2(2a_1 + 4a_2 + 6a_3 + 3a_4 + 5a_5 + 4a_6 + 3a_7 + 2a_8) \\
    &\qquad - (4a_1 + 8a_2 + 12a_3 + 6a_4 + 10a_5 + 8a_6 + 6a_7 + 3a_8)\\
    &= a_8 = \mult(\RE_8,0) \cdot s(\vph,0).
\end{align*}

\section{Quotient, cone, and final remarks}\label{sec:rmk}
In this section, we first pay special attention to a comparison of the slope and the Demailly-Lelong number on quotient singularities (Proposition~\ref{bigprop:quotient}). 
We then illustrate examples that show the estimate in Proposition~\ref{bigprop:prop_C} on $\RA_k$-singularities is sharp via the computation method on quotient singularities.
Additionally, we explore a comparison on a cone over smooth hypersurfaces in $\BP^n$. 
At the end, we address some questions on the uniform version of the comparison and maximum ratio between the Demailly--Lelong number and the slope.

\subsection{Quotient singularities}
We first give a quick overview of quotient singularities.
For more details, interested readers are referred to some standard references (see e.g. \cite[Ch.~IV]{Lamotke_book_1986}).
We say that $(X,x)$ is a quotient singularity if there exists a neighborhood $U$ that is isomorphic to a quotient $V/G$, where $G$ is a finite subgroup of $\RGL(n,\BC)$ acting linearly on an open set $V \subset \BC^n$. 
We may assume $X = \BC^n/G$, with $x$ being the image of $0$ via the quotient map $\pi: \BC^n \to \BC^n/G$.
The action $G$ induces a corresponding action on functions, given by $(g \cdot f)(p) = f(g^{-1}\cdot p)$ for $g \in G$. 
Denote by $\BC[x_1,\cdots,x_n]^G$ the ring of $G$-invariant polynomials on $\BC^n$.
Note that $\BC[x_1,\cdots,x_n]^G$ is finitely generated. 
Let $(f_i)_{i=1,\cdots,N}$ be polynomials generating $\BC[x_1,\cdots,x_n]^G$. 
The polynomial map $j: \BC^n \to \BC^N$ with components $(f_1, \cdots, f_N)$ induces a morphism 
\[
    j^\ast: \BC[z_1,\cdots,z_N] \to \BC[x_1,\cdots,x_n]^G,
    \quad F(z_1,\cdots,z_N) \overset{j^\ast}{\mapsto} F(f_1,\cdots,f_N).
\] 
The kernel of $j^\ast$ is a finitely generated ideal $\CI = \langle g_1, \cdots, g_r \rangle$, which defines an affine orbit variety $V = j(\BC^n) = \set{z \in \BC^n}{g_1(z) = \cdots = g_r(z) = 0}$.
We have $\BC[x_1,\cdots,x_n]^G \simeq \BC[z_1,\cdots,z_N]/\CI$ and a homeomorphism $J: \BC^n/G \to V$ induced by $j$.
This allows us to view $X = \BC^n/G$ as an affine algebraic variety within $\BC^N$.

\smallskip
When $(X,x)$ is isomorphic to a quotient singularity $(\BC^n/G, \pi(0))$ for $G$ a finite subgroup of $\RGL(n,\BC)$, we obtain the following comparison of the slope and the Demailly--Lelong number:  

\begin{prop}[Proposition~\ref{bigprop:quotient}]
Suppose that $(X,x)$ is isomorphic to a quotient singularity $(\BC^n/G, \pi(0))$, where $G$ is a finite subgroup of $\RGL(n,\BC)$.
Then for any germ for plurisubharmonic function $\vph: (X,x) \to \BR \cup \{-\infty\}$, 
\[
    \nu(\vph,x) \leq |G|^{n-1} \cdot s(\vph,x)
\]
where $|G|$ is the order of $G$.
\end{prop}

Before providing a proof, we first describe the expressions of the slope and the Demailly--Lelong numbers on the quotient chart. 
As we mentioned in Remark~\ref{rmk:embedding_ref_potential}, to compute the slope and the Demailly--Lelong number, one can simply choose the restriction of $\frac{1}{2} \log(\sum_{i=1}^N |z_i|^2)$ to $X$ as a reference potential. 
We also have $\pi^\ast {z_i}_{|X} = f_i$ from the previous description and identification of $X$ and $V$. 
Set $\Psi = (\sum_{i=1}^N |f_i|^2)^{1/2}$. 
Now, the slope can be expressed as
\[
    s(\vph,x) = \sup \set{\gm > 0}{\pi^\ast \vph \leq \gm \log \Psi + O(1) \text{ near } 0}.
\]
On the other hand, the Demailly--Lelong number has the following formulation
\begin{equation}\label{eq:DL_quotient}
    \nu(\vph,x) = \frac{1}{\deg(\pi)} \int_{\{0\}} \ddc \pi^\ast \vph \w (\ddc \log \Psi)^{n-1}
    = \frac{1}{\deg(\pi)} \lim_{r \to 0} \int_{\BB_r(0)} (\ddc \pi^\ast \vph) \w \lt(\ddc \log \Psi \rt)^{n-1}
\end{equation}
where $\deg(\pi)$ is the degree of $\pi$ and it equals the order of $G$.
Indeed, using the same idea as in the proof of Theorem~\ref{bigthm:intersection_and_comparison}, 
\begin{align*}
    \nu(\vph,x)
    = \lim_{r\to 0} \lim_{j \to + \infty} \lim_{\vep_1 \to 0} \cdots \lim_{\vep_{n-1} \to 0} \underbrace{\int_{\{\Psi < r\}} \ddc \vph_j \w \bigwedge_{i=1}^{n-1} \ddc \log (\Psi + \vep_i)}_{=: L_{r,j,\vep_1,\cdots,\vep_{n-1}}},
\end{align*}
where $\vph_j$ is a sequence of smooth psh functions defined near $x$ and decreasing towards $\vph$.
Since the measure $\ddc \vph_j \w \bigwedge_{i=1}^{n-1} \ddc \log (\Psi + \vep_i)$ is smooth, it puts zero mass on the singular set.
Hence,
\begin{align*}
    L_{r,j,\vep_1,\cdots,\vep_{n-1}} 
    &= \frac{1}{\deg(\pi)} \int_{\{\Psi < r\} \setminus \pi^{-1}(X^\sing)} \ddc \pi^\ast \vph_j \w \bigwedge_{i=1}^{n-1} \ddc \pi^\ast \log (\Psi + \vep_i)\\
    &= \frac{1}{\deg(\pi)} \int_{\{\Psi < r\}} \ddc \pi^\ast \vph_j \w \bigwedge_{i=1}^{n-1} \ddc \pi^\ast \log (\Psi + \vep_i).
\end{align*}
Then, taking limits inductively, one obtains \eqref{eq:DL_quotient}.

\begin{proof}[Proof of Proposition~\ref{bigprop:quotient}]
Since the action of $G$ preserves the degree, the generators can be chosen to be homogeneous. 
In \cite{Noether_1915} (see also \cite{Fleischmann_2000, Fogarty_2001}), Noether proved that generators can be taken to have degree not exceeding $|G|$, the order of $G$. 
Therefore, we claim that 
\begin{equation}\label{eq:quotient_ref_potentials}
    |G| \log |w| \leq \log \Psi + O(1)
\end{equation}
near $0$, where $w$ is the coordinate on $\BC^n$.
After a dilation, one can assume $(f_1 = \cdots = f_N = 0) \cap \BB_2(0) = \{0\}$.
We now check that our claim \eqref{eq:quotient_ref_potentials}. 
Suppose otherwise, if there is a sequence of points $w_l \to 0$ as $l \to +\infty$ and $|w_l|^{2|G|} > l \sum_i |f_i(w_l)|^2$, then $1/l \geq \sum_i r_l^{2 (\deg f_i - |G|)} |f_i(\ta_l)|^2 \geq \sum_i |f_i(\ta_l)|^2$ where $w_l = r_l \ta_l$ with $r_l = |w_k|$ and $\ta_l \in \BS^{2n-1}$.
By the compactness of $\BS^{2n-1}$, $\ta_l \to \ta \in \BS^{2n-1}$ after extracting a subsequence, and thus, $\sum_i|f_i(\ta_l)|^2 \to \sum_i |f_i(\ta)|^2 = 0$ yields a contradiction; hence the claim \eqref{eq:quotient_ref_potentials} holds.
Using Demailly's comparison \cite[Thm.~4]{Demailly_1982} and \eqref{eq:quotient_ref_potentials}, we have the following
\begin{align*}
    \nu(\vph,x) 
    &= \frac{1}{\deg(\pi)} \int_{\{0\}} (\ddc \pi^\ast \vph) \w (\ddc \log \Psi)^{n-1}\\ 
    &\leq \frac{|G|^{n-1}}{\deg(\pi)} \int_{\{0\}} (\ddc \pi^\ast \vph) \w (\ddc \log |w|)^{n-1} \\
    &= |G|^{n-2} \nu(\pi^\ast \vph, 0) 
    = |G|^{n-2} s(\pi^\ast \vph, 0)
    \leq |G|^{n-1} s(\vph,x).
\end{align*}
This finishes the proof. 
\end{proof}

\begin{eg}
Consider $X$ a product of two-dimensional ADE-singularities, i.e. $X = \prod_{i=1}^m \BC^2/G_i$ where $G_i$ is a finite subgroup of $\RSL(2,\BC)$. 
For each point $x \in X$, locally $(X,x)$ is isomorphic to $(\BC^{2m}/G_x,0)$ for a finite subgroup $G_x \in \RSL(2m,\BC)$ with $|G_x| \leq \prod_{i=1}^m |G_i|$.
As a direct consequence of Proposition~\ref{bigprop:quotient}, for every $x \in X$ and for any psh germ $\vph: (X,x) \to \BR \cup \{-\infty\}$, we obtain 
\[
    \nu(\vph,x) \leq \lt(\prod_{i=1}^m |G_i|\rt)^{2m-1} s(\vph,x).
\]
We remark that in the above case, the singular locus $X^\sing$ is not isolated.
\end{eg}

We now provide examples showing that the estimate \eqref{eq:A_k_est} is sharp:

\begin{eg}\label{eg:A_k_sharp}
Up to a change of coordinate, $\RA_k$ is isomorphic to $(xy - z^{k+1} = 0) \subset \BC^3$. 
Note that $\RA_k \simeq \BC^2/ C_{k+1}$ where $C_{k+1}$ is the cyclic group generated by 
\[
\begin{bmatrix}
    e^{\frac{2\pi\ii}{k+1}} & 0 \\
    0 & e^{\frac{-2\pi\ii}{k+1}}
\end{bmatrix}
\] 
and the map $\pi$ from the quotient chart $\BC^2$ to $\RA_k \subset \BC^3$ is $\pi(u,v) = (u^{k+1}, v^{k+1}, uv)$. 
Hence, $(k+1) \log \sqrt{|u|^2 + |v|^2} \leq \log \psi_x(\pi(u,v))$ near $0$.

\smallskip 
Consider a psh function $\vph$ on $\RA_k$ given by the restriction of the psh function $\log (|x| + |z|^{k+1})$ on $\BC^3$. 
Set $x_0$ be the singular point. 
Then 
\begin{align*}
    \nu(\vph, x_0) &= \frac{1}{k+1} \lim_{r \to 0} \int_{\BB_r(0)} \ddc \log (|u^{k+1}| + |uv|^{k+1}) \w \frac{1}{2} \ddc \log (|u|^{2k+2} + |v|^{2k+2} + |uv|^2) \\
    &= \frac{1}{2} \underbrace{\lim_{r \to 0} \int_{\BB_r(0)} \ddc \log |u| \w \ddc \log (|u|^{2k+2} + |v|^{2k+2} + |uv|^2)}_{=: I} \\
    &\qquad + \frac{1}{2} \lim_{r \to 0} \int_{\BB_r(0)} \ddc \log(1 + |v|^{k+1}) \w \ddc \log (|u|^{2k+2} + |v|^{2k+2} + |uv|^2). 
\end{align*}
The second term equals zero. 
Hence, we only need to deal with the first term $I$.
We have
\[
    I = \lim_{r \to 0} \int_{\BB_r(0)} [(u=0)] \w \ddc \log( |u|^{2k+2} + |v|^{2k+2} + |uv|^2)
    = \lim_{r \to 0} \int_{|v|<r} \ddc \log |v|^{2k+2} = 2k+2
\]
and thus,  
\begin{equation}\label{eq:nu_shape}
    \nu(\vph, x_0) = k+1. 
\end{equation}
On the other hand, 
\[
    s(\vph, x_0) = \liminf_{\RA_k \setminus \{0\} \ni p = (x,y,z) \to 0} \frac{\log (|x| + |z|^{k+1})}{\frac{1}{2} \log (|x^2| + |y|^2 + |z|^2)}.
\]
Considering sequence of points $p_n = (1/n, 0, 0) \in \RA_k \setminus \{0\}$, we then obtain
\[
    s(\vph, x_0) \leq \lim_{n\to + \infty} \frac{\log(1/n)}{\frac{1}{2}\log (1/n^2)} = 1.
\]
Combining with \eqref{eq:A_k_est} and \eqref{eq:nu_shape}, we obtain $s(\vph, x_0) = 1$, and this ensures that the estimate in \eqref{eq:A_k_est} is sharp.
\end{eg}

\subsection{Cone singularities}
We now turn our attention to an isolated singularity that can be resolved through a single blowup with an irreducible exceptional locus.

\begin{cor}\label{cor:one_blowup_one_comp}
Let $x$ be an isolated singularity in $X$. 
Suppose that there exists a resolution of singularity $\pi: \wX \to X$ by a single blowup with $x$, and the exceptional divisor $E$ over $x$ is irreducible.  
Then for any psh germ $\vph: (X,x) \to \BR \cup \{-\infty\}$, 
\[
    \mult(X,x) \cdot s(\vph,x) = \nu(\vph,x).
\]
\end{cor}  

\begin{proof}
According to Siu's decomposition, $\ddc \pi^\ast \vph = a [E] + R$ for some $a \geq 0$ and positive closed $(1,1)$-current $R$ with generic Lelong number zero along $E$. 
Set $m \in \BN^\ast$ such that $\pi^{-1} \fm_{X,x} = \CO_\wX(-mE)$. 
By Theorem~\ref{bigthm:intersection_and_comparison}, we obtain 
\[
    \nu(\vph,x) = D \cdot (-mE)^{n-1} = \frac{a}{m} \cdot (mE) \cdot (-mE)^{n-1} = s(\vph,x) \cdot (mE) \cdot (-mE)^{n-1}.  
\]
Moreover, we also have $\mult(X,x) = \nu(\log \psi_x, x) = (mE) \cdot (-mE)^{n-1}$ and this completes the proof. 
\end{proof}

\begin{eg}
Consider $D = (f = 0) \subset \BP^n$ a smooth irreducible hypersurface of degree $d$ and let $X = \CC(D) := (f = 0) \subset \BC^{n+1}$ be the cone over $D$. 
The vertex $0$ is the only singular point of $X$. 
Note that $X$ can be desingularized by taking a single blowup of the origin $\pi: Y \to X$ with $\pi^{-1}\fm_0 = \CO_Y(-E)$ and $\mult(X,0) = d$. 
From Corollary~\ref{cor:one_blowup_one_comp}, for any psh germ $\vph: (X,0) \to \BR \cup \{-\infty\}$, 
\[
    \mult(X,0) \cdot s(\vph,0) = d \cdot s(\vph,0) = \nu(\vph,0).
\]
\end{eg}

\subsection{Questions}
It is natural to wonder whether a uniform version of the comparison still holds, namely, a comparison independent of both the base point $x$ and the psh germ $\vph$ at $x$.
Precisely, we address the following question:
\begin{ques}\label{ques:general_comparison}
Let $X$ be an $n$-dimensional
locally irreducible reduced
complex analytic space. 
Fix a compact subset $K \subset X$. 
Does there exist a uniform constant $C_K>0$ such that for all $x \in K$, for all psh germs $\vph: (X,x) \to \BR \cup \{-\infty\}$, 
\[
    \nu(\vph,x) \leq C_K \cdot \mult(X,x) \cdot s(\vph,x) ?
\]
In addition, can one find a bounded function $f(x)$ expressed by a certain algebraic invariant of $(X,x)$ such that $\nu(\vph,x) \leq f(x) \cdot \mult(X,x) \cdot s(\vph,x)$?
\end{ques}

It is also interesting to ask if the supremum of the ratio between the Demailly--Lelong number and the slope can be achieved by some psh functions with analytic singularities: 

\begin{ques}\label{ques:max_ratio}
Does there exists a psh germ $\psi: (X,x) \to \BR \cup \{-\infty\}$ with an isolated analytic singularities and non-zero slope at $x$ such that
\[
    \frac{\nu(\psi,x)}{s(\psi,x)} 
    = \sup\set{\frac{\nu(\vph,x)}{s(\vph,x)}}{\text{for all psh germs } \vph: (X,x) \to \BR \cup \{-\infty\} \text{ with }s(\vph,x) \neq 0}?
\]
\end{ques}

Finally, we make a remark on the connection with algebraic quantities. 

\begin{rmk}
We recall the notion of mixed multiplicities (cf. \cite[\S~2]{Teissier_1973}, \cite[Sec.~1.6.B]{Lazarsfeld_book_1}, \cite[Sec.~4.4]{Boucksom_Favre_Jonsson_2014}). 
The (Hilbert--Samuel) multiplicity of an $\fm_{X,x}$-primary ideal $\fa \subset \CO_{X,x}$ is defined as the limit
\[
    \e(\fa) = \lim_{k \to +\infty} \frac{n!}{k^n} \length(\CO_{X,x}/\fa^k),
\]where $\length(\CO_{X,x}/\fa^k)$ is the length of the Artinian ring $\CO_{X,x}/\fa^k$.
The mixed (Hilbert--Samuel) multiplicities of any two $\fm_{X,x}$-primary ideals $\fa_1$, $\fa_2$ are a sequence of $(n+1)$ integers $(\e(\fa_1^{[i]}; \fa_2^{[n-i]}))_{i = 0, \cdots, n}$ such that for all $p,q \in \BZ_+$
\[
    \e(\fa_1^p \cdot \fa_2^q)
    = \sum_{i=0}^n \binom{n}{i} \e(\fa_1^{[i]}; \fa_2^{[n-i]}) p^{i} q^{n-i}.
\]

\smallskip
Fix an $\fm_{X,x}$-primary ideal $\CI \subset \CO_{X,x}$ and take $\vph_\CI = \frac{1}{2} \log \sum_i |g_i|^2$ where $(g_i)_i$ are local generators of $\CI$. 
Then by Theorem~\ref{bigthm:intersection_and_comparison} and \cite[p.~92]{Lazarsfeld_book_1} (cf. \cite[Rmk.~A.5]{BBEGZ_2019} and \cite[Lem.~1.2]{Demailly_2009}), we have 
\[
    \nu(\vph_\CI,x) = \e_1(\CI,x) 
    := \e(\CI^{[1]};\fm_{X,x}^{[n-1]}).
\]
On the other hand, the slope can be expressed as the order of vanishing; namely, 
\[
    s(\vph_\CI,x)
    = \overline{\ord}_x(\CI) := \lim_{k \to +\infty}\frac{1}{k} \ord_x(\CI^k)
\]
with $\ord_x(\CI^k) := \set{l \in \BN}{\CI^k \subset \fm_{X,x}^l}$.
Indeed, taking $\pi: \wX \to X$ such that $\pi^{-1}\CI = \CO_\wX(-\sum_i a_i E_i)$ and $\pi^{-1}\fm_{X,x} = \CO_\wX(-\sum_i m_i E_i)$, by definition, we have 
\begin{align*}
    \ord_x(\CI^k) 
    &= \max\set{l \in \BN}{\CO_\wX(-k \sum_i a_i E_i) \subset \CO_\wX(-l \sum_i m_i E_i)} \\
    &= \max\set{l \in \BN}{k a_i \geq l m_i\ \, \text{for all $i$}}.
\end{align*}
Hence, we obtain $\ord_x(\CI^k)/k \leq \min_i (a_i/m_i) = s(\vph_\CI,x) \leq (\ord_x(\CI^k) + 1)/k$.

\smallskip
If Question~\ref{ques:max_ratio} holds, then finding the supremum of the ratio $\nu/s$ becomes looking for the following more algebro-geometric invariant
\[
    \sup\set{\frac{\e_1(\CI,x)}{\overline{\ord}_x(\CI)}}{\text{for all $\fm_{X,x}$-primary ideal } \CI \subset \CO_{X,x}}.
\]
\end{rmk}

\bibliographystyle{smfalpha_new}
\bibliography{biblio}

\end{document}